\documentclass[runningheads]{llncs}

\usepackage{amssymb, amsmath, amsfonts}
\usepackage{graphicx}
\usepackage{xcolor}
\usepackage{esvect}
\usepackage{hyperref}
\usepackage{todonotes}
\usepackage{xspace}
\usepackage{multirow}
\usepackage[lined,algoruled,commentsnumbered,noend]{algorithm2e}
\usepackage{marvosym}

\newtheorem{Assumption}{Assumption}

\newcommand{\Z}{\mathbb{Z}}
\newcommand{\R}{\mathbb{R}}

\DeclareMathOperator{\st}{s.t.}
\DeclareMathOperator{\conv}{conv}

\newcommand{\BPs}{BPs\xspace}
\newcommand{\MIBLOP}{MIBLP\xspace}
\newcommand{\MIBLOPs}{MIBLPs\xspace}

\newcommand{\VFR}{VFR\xspace}
\newcommand{\HPR}{HPR\xspace}
\newcommand{\BB}{B\&B\xspace}
\newcommand{\BC}{B\&C\xspace}

\newcommand{\IBNP}{IBNP\xspace}
\newcommand{\IBNPs}{IBNPs\xspace}
\newcommand{\SOCP}{SOCP\xspace}
\def\yy{\hat{y}}
\def\xh{\hat{x}}
\newcommand{\exSymb}[1]{\hat{#1}}
\newcommand{\linq}{g}
\newcommand{\DCs}{DCs\xspace}
\newcommand{\DC}{DC\xspace}
\newcommand{\MIX}{\texttt{MIX++}\xspace}

\newcommand{\textoverline}[1]{$\overline{\mbox{#1}}$}
\newcommand{\ccrHPR}{\textoverline{\HPR}\xspace}

\begin{document}

\title{SOCP-based disjunctive cuts for a class of 
integer nonlinear bilevel programs\thanks{
This research was funded in whole, or in part, by the Austrian Science Fund (FWF) [P 35160-N]. For the purpose of open access, the author has applied a CC BY public copyright licence to any Author Accepted Manuscript version arising from this submission.
It is also supported by the Johannes Kepler University Linz, Linz Institute of Technology (Project LIT-2019-7-YOU-211) and the JKU Business School. J. Lee was supported on this project by ESSEC and by ONR grant N00014-21-1-2135. }}

\titlerunning{SOCP-based disjunctive cuts for a class of \IBNP}

\author{Elisabeth Gaar$^{(\text{\Letter})}$\inst{1}\orcidID{0000-0002-1643-6066} \and
Jon Lee \inst{2}\orcidID{0000-0002-8190-1091} \and \\
Ivana Ljubi\'c\inst{3}\orcidID{0000-0002-4834-6284} \and
Markus Sinnl\inst{1,4}\orcidID{0000-0003-1439-8702} \and \\
K\"ubra Tan{\i}nm{\i}\c{s} \inst{1}\orcidID{0000-0003-1081-4182} 
}

\authorrunning{E. Gaar et al.}
\institute{
	Institute of Production and Logistics Management, Johannes Kepler University Linz, Austria, 
	\email{\{elisabeth.gaar, markus.sinnl, kuebra.taninmis\_ersues\}@jku.at}\\
	\and
	University of Michigan, Ann Arbor, Michigan, USA,
	\email{jonxlee@umich.edu}
	\and
	ESSEC Business School of Paris, France, \email{ljubic@essec.edu}
	\and
	JKU Business School, Johannes Kepler University Linz, Austria
}

\maketitle              

\begin{abstract}
We study a class of bilevel integer programs with second-order cone constraints at the upper level and a convex quadratic objective and linear constraints at the lower level. We develop disjunctive cuts to separate bilevel infeasible points using a second-order-cone-based cut-generating procedure. 
To the best of our knowledge, this is the first time disjunctive cuts are studied in the context of discrete bilevel optimization.
Using these disjunctive cuts, we establish a branch-and-cut algorithm for the problem class we study, and a cutting plane method for the problem variant with only binary variables.
We present a preliminary computational study on instances with no second-order cone constraints at the upper level and a single linear constraint at the lower level. Our study demonstrates that both our approaches outperform a state-of-the-art generic solver for mixed-integer bilevel linear programs that is able to solve a linearized version of our test instances, 
where the non-linearities are linearized in a McCormick fashion.
	\keywords{bilevel optimization  \and disjunctive cuts \and conic optimization \and nonlinear optimization \and branch-and-cut.}
\end{abstract}

\section{Introduction}

Bilevel programs (\BPs) are challenging hierarchical optimization problems, in which the feasible solutions of the
upper level problem depend on the optimal solution of the lower level problem. \BPs allow to model 
two-stage two-player Stackelberg games in which two 
rational players
(often called \emph{leader} and \emph{follower}) compete in a sequential fashion. 
\BPs have applications in many different domains such as machine learning
\cite{agor2019feature}, logistics 
\cite{fontaine2020population}, revenue 
management 
\cite{labbe2016bilevel}, the energy sector 
\cite{grimm2021optimal,plein2021bilevel}
and portfolio optimization \cite{gonzalez2021global}. For more details about 
\BPs see, e.g., 
the book by Dempe and Zemkoho~\cite{dempe2020bilevel} and two recent 
surveys~\cite{kleinert2021survey,smith2020survey}.

In this work, 
we consider the following 
integer nonlinear bilevel programs with convex leader and follower objective functions (\IBNPs)
\begin{subequations}
\label{bilevel}
\begin{align}
&\min ~c'x + d'y \label{eq:objective}\\ 
\st~ &Mx+Ny \geq h\\
&\tilde{M}x+\tilde{N}y -  \tilde{h} \in \mathcal{K} \label{eq:conic}\\
&y\in \arg \min\left\{
q(y)  : Ax+By \geq f,~ y \in 
\mathcal{Y}, ~y \in \mathbb Z^{n_2} 
\right\}\\
&x \in \mathbb Z^{n_1},
\end{align}
\end{subequations}
where
the decision variables $x$ and $y$ are of dimension $n_1$ and $n_2$, 
respectively, and $n=n_1+n_2$.
%
Moreover, we have
$c \in \R^{n_1}$, 
$d \in \R^{n_2}$, 
$M \in \R^{m_1 \times n_1}$, 
$N \in \R^{m_1 \times n_2}$, 
$h \in \R^{           m_1}$, 
$\tilde{M} \in \R^{\tilde{m}_1 \times n_1}$, 
$\tilde{N} \in \R^{\tilde{m}_1 \times n_2}$, 
$\tilde{h} \in \R^{\tilde{m}_1}$, 
$A \in \Z^{m_2 \times n_1}$, 
$B \in \Z^{m_2 \times n_2}$, and 
$f \in \Z^{m_2 }$. We assume that each row of $A$ and $B$ has at least one non-zero entry and 
the constraints $Ax+By \geq f$ are referred to as \emph{linking constraints}.
Furthermore, 
$q(y)$ is a convex 
quadratic function of the form
$q(y) = y'Ry + \linq'y$ with $R=V'V$ and 
$V \in \R^{n_3 \times n_2}$ with $n_3 \leq n_2$,
$\mathcal{K}$ is a given
 cross-product of
second-order cones, and $\mathcal{Y}$ is a polyhedron.

Note that even though we formulate the objective function \eqref{eq:objective} as linear, 
we can actually consider any convex objective function which can be represented as a second-order cone constraint and 
whose optimal value is integer when $(x,y) \in \mathbb{Z}^{n}$ (e.g., a convex quadratic polynomial with integer coefficients). To do so, we can use an epigraph reformulation to transform it into a problem of the form~\eqref{bilevel}.
%

Our work considers the \emph{optimistic} case of bilevel optimization. This means that whenever there are multiple optimal solutions for the follower problem, the one which is best for the leader is chosen, see, e.g., \cite{loridan1996weak}. We note that already mixed-integer bilevel linear programming (\MIBLOP) is $\Sigma_2^p$-hard~\cite{Lodi-et-al:2014}.

The \emph{value function reformulation} (\VFR) of the bilevel model~\eqref{bilevel} is given 
as
\begin{subequations}
	\label{vfr}
	\begin{align}
	&\min  ~c'x + d'y \\
	\st~ &Mx+Ny \geq h \label{ineq:linear}\\
	&\tilde{M}x+\tilde{N}y -  \tilde{h} \in \mathcal{K} \label{ineq:tilde}\\
	& Ax+By \geq f \label{ineq:linking}\\
	&q(y) \leq \Phi(x) \label{eq:inequValueFunction}\\
	& y \in \mathcal{Y} \label{ineq:yinY} \\
	&(x,y) \in \mathbb Z^n,
	\label{ineq:xyint}
	\end{align}
\end{subequations}
where the so-called \emph{value function} $\Phi(x)$ of the \emph{follower problem}
\begin{equation}
\Phi(x) = {\rm min}\left\{
q(y)  : Ax+By \geq f,~  y \in \mathcal{Y}, 
~ y \in \mathbb Z^{n_2}
\right\} \label{eq:follower}
\end{equation}
is typically non-convex and non-continuous. Note that the \VFR is equivalent to the original bilevel model~\eqref{bilevel}.
The \emph{high point relaxation} (\HPR) is obtained when 
dropping~\eqref{eq:inequValueFunction}, i.e., the optimality condition of $y$ 
for the follower problem, from the \VFR~\eqref{vfr}. 
We denote the continuous relaxation (i.e., replacing the integer constraint \eqref{ineq:xyint} with the corresponding variable bound constraints)
of the \HPR as \ccrHPR.
A solution $(x^*,y^*)$ is called \emph{bilevel infeasible}, if it is feasible for  \ccrHPR, but not feasible for the original bilevel model~\eqref{bilevel}.

\paragraph{Our Contribution}
Since
the seminal work of Balas~\cite{BalasDP}, and more intensively in the past three decades, disjunctive cuts (\DCs) have been successfully exploited for solving mixed-integer (nonlinear) programs (MI(N)LPs)~\cite{balas2018disjunctive}.  While there is a plethora of work on using \DCs for MINLPs~\cite{belotti2011disjunctive}, we are not aware of any previous applications of \DCs for solving \IBNPs.  
In this work we demonstrate how   
\DCs can be used within in a branch-and-cut (\BC) algorithm 
to solve \eqref{bilevel}. 
This is the first time that \DCs are used to separate bilevel infeasible points, using a cut-generating procedure based on second-order cone programming (\SOCP).   
Moreover, we also show that our \DCs can  be used in a finitely-convergent cutting plane procedure for 0-1 \IBNPs, where the \HPR is solved to optimality before separating bilevel infeasible points.
Our computational study is conducted on instances in which the follower minimizes a convex quadratic function, subject to a covering constraint linked with the leader. We compare the proposed \BC and cutting plane approaches  with a state-of-the-art solver for \MIBLOPs (which can solve our instances after applying linearization in a McCormick fashion), and show that the latter one is outperformed by our new DC-based methodologies. 


\paragraph{Literature Overview}
For \MIBLOPs with integrality restrictions on (some of) the follower variables, 
state-of-the-art computational methods are usually based on \BC (see, e.g.,
\cite{fischetti2016intersection,fischetti2017new,fischetti2018use,Tahernejad2020}). Other interesting 
concepts are based on multi-branching, 
see~\cite{wang2017watermelon,xu2014exact}. 

Considerably less results are available for non-linear \BPs, and in particular with integrality restrictions at the lower level. In~\cite{mitsos2008global}, Mitsos et al.\ propose a general approach for non-convex follower problems 
which solves nonlinear optimization problems to compute upper and lower bounds in an iterative fashion. 
In a series of papers on the so-called \emph{branch-and-sandwich} approach,  
tightened bounds on the optimal value function and on the leader's objective function value are calculated~\cite{KleniatiAdjiman2014b,kleniati2015generalization,KleniatiAdjiman2014a}. 
 A solution algorithm for mixed-\IBNPs proposed in~\cite{lozano2017value} by Lozano and Smith
approximates the value function by dynamically inserting additional variables and big-M type of constraints. 
Recently, Kleinert et al.~\cite{schmidt2021} considered bilevel problems with a mixed-integer convex-quadratic upper level and a continuous convex-quadratic lower level. The method is based on outer approximation after the problem is reformulated into a single-level one using the strong duality and convexification.  
In \cite{byeon2022benders}, Byeon and Van Hentenryck develop a solution algorithm for bilevel problems, where the leader problem can be modeled as a mixed-integer second-order conic problem and the follower problem can be modeled as a second-order conic problem. The algorithm is based on a dedicated Benders decomposition method.
In~\cite{weninger2020}, Weninger et al.\ propose a methodology that can tackle any kind of a MINLP at the upper level which can be handled by an off-the-shelf solver. The mixed-integer lower level problem has to be convex, bounded, and satisfy Slater’s condition for the continuous variables. This exact method is derived from a previous approach proposed in~\cite{YueGZY19} by Yue et al.\ for finding bilevel feasible solutions.
For a more detailed  overview of the recent literature on computational bilevel optimization we refer an interested reader to~\cite{CerulliThesis,kleinert2021survey,smith2020survey}.

The only existing application of  DCs in the context of bilevel \emph{linear} optimization is by Audet et al., \cite{audet} who derive DCs from LP-complementarity conditions. In~\cite{judice}, J{\'{u}}dice et al.\ exploit a similar idea for solving mathematical programs with equilibrium constraints.   
DCs are frequently used for solving MINLPs  (see, e.g., \cite{balas2018disjunctive}, and the many references therein, and 
\cite{MIQCP,MIQCPproj}). 
In~\cite{Kilinc2014}, K{\i}l{\i}n{\c{c}}-Karzan 
and Y{\i}ld{\i}z derive closed-form expressions for inequalities describing the convex hull of a two-term disjunction applied to the second-order cone. 

\section{Disjunctive Cut Methodology \label{sec:disj}}
The aim of this section is to derive \DCs for the bilevel model~\eqref{bilevel} with the help of SOCP, so we want to
derive \DCs that are able to separate 
bilevel infeasible points 
from the convex hull of bilevel feasible ones. 
Toward this end, we assume throughout this section that 
we have a second-order conic convex set $\mathcal{P}$, such that 
the set of feasible solutions of the \VFR is a subset of $\mathcal{P}$, and such that $\mathcal{P}$ is a subset of the set of feasible solutions of the \ccrHPR. 
This implies that $\mathcal{P}$ fulfills 
\eqref{ineq:linear},  \eqref{ineq:tilde}, \eqref{ineq:linking} and \eqref{ineq:yinY} and potentially already some \DCs.
Moreover, we assume that $(x^*,y^*)$ is a bilevel infeasible point in $\mathcal{P}$.
The point $(x^*,y^*)$ is an \emph{extreme point} of $\mathcal{P}$, if it is not a convex combination of any other two points of $\mathcal{P}$.

\subsection{Preliminaries}

For clarity of exposition in what follows, we  consider only one linking  
constraint of problem~\eqref{bilevel}, i.e., $m_2 = 1$ and thus $A = a'$ and $B = b'$ 
for some $a \in \Z^{n_1}$, $b \in \Z^{n_2}$ and $f \in \Z$. Note however that our methodology can be generalized for multiple linking constraints leading to one additional disjunction for every additional linking constraint. 
Moreover, our \DCs need the following assumptions.



\begin{Assumption}\label{as:bound1}
All variables are bounded in the \HPR and 
$\mathcal{Y}$ is bounded.
\end{Assumption}
Assumption~\ref{as:bound1} ensures that the \HPR is bounded. 
We note that in a bilevel-context already for the linear case of \MIBLOPs,
unboundedness of the \ccrHPR does not imply anything for the original problem, 
all three options (infeasible, unbounded, and existence of an optimum) are possible. For more details see, e.g., \cite{fischetti2018use}. 

\begin{Assumption}
For every $x$, such that there exists a $y$ with $(x,y)$ being feasible for the \ccrHPR, 
the follower problem~\eqref{eq:follower} is feasible.
\end{Assumption}

\begin{Assumption}
\label{as:relHPRsolvable}
\ccrHPR has a feasible solution satisfying its nonlinear constraint  \eqref{ineq:tilde}  strictly, and its dual has a  feasible solution. 
\end{Assumption}
Assumption~\ref{as:relHPRsolvable} ensures that we have strong duality between \ccrHPR
and its dual, and so we can solve the \ccrHPR (potentially with added cuts) to
arbitrary accuracy.




\subsection{Deriving Disjunctive Cuts}
\label{sec:DerivingCuts}

To derive \DCs we first examine bilevel feasible points. 
It is easy to see and also follows from the results by Fischetti et al.~\cite{fischetti2017new}, that 
for any $\hat y \in \mathcal{Y} \cap \mathbb{Z}^{n_2}$
the set 
$$
S(\hat y) = \{ (x,y) : a'x \geq f - b'\hat{y},~ q(y) > q(\hat{y})  \}
$$
does not contain any 
bilevel feasible solutions, as for any $(x,y) 
\in S(\hat y)$ clearly $\hat y$ is a better follower solution than $y$ for $x$. 
Furthermore, due to the integrality of our variables 
and 
of $a$ and $b$, the  
extended set 
\[
S^+(\hat y) = \{ (x,y) : a'x \geq f - b'\hat{y}-1,~ q(y) \geq q(\hat{y})  
\}
\]
does not contain any bilevel feasible solutions in its interior, because any 
bilevel feasible solution in the interior of $S^+(\hat y)$ 
is in $S(\hat y)$. 
Based on this observation intersection cuts have been derived in~\cite{fischetti2017new}, however $S^+(\hat y)$ is not convex in our case, so we turn our attention to \DCs.
As a result, for any  $\hat y \in \mathcal{Y}  \cap \mathbb{Z}^{n_2} $ any bilevel 
feasible solution is in the disjunction $\mathcal{D}_1(\hat{y}) \vee \mathcal{D}_2(\hat{y})$, where
\begin{align*}
\mathcal{D}_1(\hat{y}) : a'x \leq f - b'\hat{y}  - 1 
\qquad \text{ and } \qquad
\mathcal{D}_2(\hat{y}) : q(y) \leq q(\hat{y}).
\end{align*}
%
To find a \DC,
we want to generate valid linear inequalities
for
\begin{equation}
\label{eq:setOfDisjunction}
\left\{(x,y)\in \mathcal{P} : \mathcal{D}_1(\hat{y})\right\}
\vee
\left\{(x,y)\in \mathcal{P} : \mathcal{D}_2(\hat{y})\right\},
\end{equation}
so in other words we want to find a valid linear inequality that separates the bilevel infeasible solution $(x^*,y^*)$ from 
\begin{equation*}
\mathcal{D}(\hat{y}) = 
\conv \left(
\left\{(x,y)\in \mathcal{P} : \mathcal{D}_1(\hat{y})\right\}
\cup
\left\{(x,y)\in \mathcal{P} : \mathcal{D}_2(\hat{y})\right\}
\right).
\end{equation*}

Toward this end, we first derive a formulation of $\mathcal{P}$.
If we have already generated
some \DCs of the form $\alpha'x +\beta'y \geq \tau$, then they create a 
bunch of constraints
$ \mathcal{A}x + \mathcal{B} y \ge \mathcal{T}$. 
We take these cuts, together with $Mx+Ny \geq h$ and $a'x+b'y \geq f$ and also 
$y \in \mathcal{Y}$, which can be represented as $\mathcal{C} y \ge \mathcal{U}$,
and we bundle them all together as
\begin{align}
&\bar{M}x + \bar{N}y \geq \bar{h}, 
\label{ineq:bar}
\end{align}
such that $\mathcal{P}$ is represented by \eqref{ineq:bar} and \eqref{ineq:tilde}, and where
\begin{align*}
\bar M = \begin{pmatrix}
M \\
a' \\
\mathcal{A} \\
0 
\end{pmatrix}
,
\qquad
\bar N = \begin{pmatrix}
N \\
b' \\
\mathcal{B} \\
\mathcal{C}
\end{pmatrix}
,
\qquad
\bar h = \begin{pmatrix}
h \\
f \\
\mathcal{T} \\
\mathcal{U}
\end{pmatrix}
.
\end{align*}

The representation of $\mathcal{D}_1(\hat{y})$ is straightforward. 
It is convenient to write $\mathcal{D}_2(\hat{y})$ in 
SOCP-form using a standard technique.
Indeed, $\mathcal{D}_2(\hat{y})$ is equivalent to the standard 
second-order (Lorentz) cone constraint
$z^0 \geq \left\| (z^1,z^2) \right\|$ with 
\begin{align*}
z^0 = \frac{1-\left(\linq'y- q(\hat{y}) \right)}{2},
\qquad 
z^1 = Vy, 
\qquad
z^2 = \frac{1+\left(\linq'y- q(\hat{y})\right)}{2}.
\end{align*}
Because $z^0$, $z^1$ and $z^2$ are linear in $y$, we can as well write it in the form
\begin{align}
\label{con:newAtilde}
& \tilde{D}y - \tilde{c} \in \mathcal{Q}, 
\end{align}
where $\mathcal{Q}$ denotes a standard second-order cone,  which is self dual, and  
\[
\tilde{D} = \left(
\begin{array}{c}
-\frac{1}{2}\linq' \\[3pt]
V \\[3pt]
\frac{1}{2}\linq'
\end{array}
\right)
\qquad \text{ and } \qquad
\tilde{c} = \left(
\begin{array}{c}
\frac{-1-q(\hat{y})}{2} \\[3pt]
0
\\[3pt]
\frac{-1+q(\hat{y})}{2}
\end{array}
\right).
\]

We employ a scalar  
 dual multiplier  
 $\sigma$ for the constraint $\mathcal{D}_1(\hat{y})$ and we 
 employ a vector $\rho \in \mathcal{Q^*}$  of dual multipliers for 
 the constraint~\eqref{con:newAtilde}, representing~$\mathcal{D}_2(\hat{y})$. 
 Furthermore, we employ two vectors
 $\bar{\pi}_k$, $k=1,2$,
 of dual multipliers for the constraints \eqref{ineq:bar}
 and we employ two vectors
 $\tilde{\pi}_k$, $k=1,2$,
 of dual multipliers for the constraints \eqref{ineq:tilde}, both together representing $\mathcal{P}$.
Then
every $(\alpha,\beta,\tau)$ corresponding to
a valid linear inequality $\alpha'x +\beta'y \geq \tau$
for $\mathcal{D}(\hat{y})$ corresponds to a solution of
\begin{subequations}
\begin{align}
&\alpha' = \bar{\pi}_1'\bar{M} + \tilde{\pi}_1'\tilde{M} +\sigma 
{a}'\label{eq1} \\
&\alpha' = \bar{\pi}_2'\bar{M} + \tilde{\pi}_2'\tilde{M} \label{eq2} \\
&\beta' = \bar{\pi}_1'\bar{N} + \tilde{\pi}_1'\tilde{N}  \label{eq3} \\
&\beta' = \bar{\pi}_2'\bar{N} + \tilde{\pi}_2'\tilde{N}  +  \rho' \tilde{D} 
\label{eq4} \\
&\tau \leq \bar{\pi}_1'\bar{h} + \tilde{\pi}_1'\tilde{h} + 
\sigma(f-1-b'\hat{y}) \label{eq5} \\
&\tau \leq \bar{\pi}_2'\bar{h} + \tilde{\pi}_2'\tilde{h} + \rho' \tilde{c} 
\label{eq6} \\
&\bar{\pi}_1 \geq 0,~ \bar{\pi}_2 \geq 0,~
\tilde{\pi}_1  \in \mathcal{K^*},~ \tilde{\pi}_2  \in \mathcal{K^*},~
\sigma\leq 0,~ \rho \in \mathcal{Q^*}, \label{eq7}
\end{align}
\end{subequations}
 where $\mathcal{K}^*$ and $\mathcal{Q^*}$ are the dual cones of  $\mathcal{K}$ and $\mathcal{Q}$, respectively
 (see, e.g., Balas~\cite[Theorem 1.2]{balas2018disjunctive}).

To attempt to generate a valid inequality for $\mathcal{D}(\hat{y})$ that is violated by
the bilevel infeasible solution
$(x^*,y^*)$, we solve
\begin{align*}
&\max\ \tau - \alpha'x^* -\beta'y^* \tag{CG-SOCP} \label{CG-SOCP}\\
\st &~\eqref{eq1}\mbox{--}\eqref{eq7}.
\end{align*}
A positive objective value for a feasible
$(\alpha,\beta,\tau)$ corresponds to
a valid linear inequality $\alpha'x +\beta'y \geq \tau$
for $\mathcal{D}(\hat{y})$
violated by
$(x^*,y^*)$, i.e. the inequality gives a \DC separating $(x^*,y^*)$ from  $\mathcal{D}(\hat{y})$. 

Finally, we need to deal with the fact that the feasible region of \eqref{CG-SOCP}
is a cone. So \eqref{CG-SOCP} either has its optimum at the origin (implying that
$(x^*,y^*)$ cannot be separated), or \eqref{CG-SOCP} is unbounded, implying that
there is a violated inequality, which of course we could scale by any positive
number  so as to make the violation as large as we like. The standard remedy for
this is to introduce a normalization constraint to \eqref{CG-SOCP}.
A typical good choice (see~\cite{Fischetti2011})
is to impose $\|(\bar{\pi}_1, \bar{\pi}_2, \tilde{\pi}_1, \tilde{\pi}_2, 
\sigma, \rho  )\|_1\ \leq 1$,
but in our context, because we are using a conic solver,
we can more easily and efficiently  impose
 $\|(\bar{\pi}_1, \bar{\pi}_2, \tilde{\pi}_1, \tilde{\pi}_2, \sigma, \rho  
 )\|_2\ \leq 1$,
 which is just one constraint for a conic solver. Thus, we will from now on consider normalization as part of \eqref{CG-SOCP}.

To be able to derive \DCs we make the following additional assumption.
\begin{Assumption}
\label{as:CGSOCPsolvable}
The dual of \eqref{CG-SOCP} has a feasible solution in its interior
and we have an exact solver for \eqref{CG-SOCP}.
\end{Assumption}

\noindent We have the following theorem, which allows us to use \DCs in solution methods.
 
 \begin{theorem}\label{tm:getCut}
Let $\mathcal{P}$ be a second-order conic convex set, such that the set of feasible solutions of the \VFR is a subset of $\mathcal{P}$, and such that $\mathcal{P}$ is a subset of the set of feasible solutions of the \ccrHPR.
 Let $(x^*,y^*)$ be bilevel infeasible and be an extreme point of $\mathcal{P}$. 
 Let $\hat y$ be a feasible solution to the follower problem for $x=x^*$
(i.e., $\hat y \in \mathcal{Y} \cap \mathbb Z^{n_2}$ 
and $a'x^*+b' \hat{y} \geq f$) such that $q(\hat{y}) < q(y^*)$.
 
 Then there is a \DC that separates $(x^*,y^*)$ from  $\mathcal{D}(\hat{y})$ and it can be obtained by solving \eqref{CG-SOCP}.
 \end{theorem}
 \begin{proof}
 Assume that there is no cut that separates $(x^*,y^*)$ from  $\mathcal{D}(\hat{y})$, then $(x^*,y^*)$ is in $\mathcal{D}(\hat{y})$. 
 However, due to the definition of $\hat{y}$, 
 the point $(x^*,y^*)$ does not fulfill $\mathcal{D}_1(\hat{y})$ and does not fulfill $\mathcal{D}_2(\hat{y})$.
 Therefore, in order to be in $\mathcal{D}(\hat{y})$, the point $(x^*,y^*)$ must be a convex combination of one point in $\mathcal{P}$ that fulfills $\mathcal{D}_1(\hat{y})$, and 
 another point in $\mathcal{P}$ that fulfills $\mathcal{D}_2(\hat{y})$. This is not possible due to the fact that $(x^*,y^*)$ is an extreme point of $\mathcal{P}$.
 
 Thus, there is a cut that separates $(x^*,y^*)$ from  $\mathcal{D}(\hat{y})$. By construction of \eqref{CG-SOCP} and due to   Assumption~\ref{as:CGSOCPsolvable}, we can use \eqref{CG-SOCP}  to find it.\qed
 \end{proof}
 

 Note that there are two reasons why a feasible \ccrHPR solution $(x^*,y^*)$ 
 is bilevel infeasible: it is not integer or $y^*$ is not the optimal follower response for $x^*$.
 Thus, in the case that $(x^*,y^*)$ is integer, there is a better follower response $\tilde{y}$ for $x^*$. Then Theorem~\ref{tm:getCut} with  
 $\hat{y} = \tilde{y}$ implies that $(x^*,y^*)$ can be separated
 from  $\mathcal{D}(\hat{y})$.
 We present solution methods based on this observation in Section~\ref{sec:intcut}.
 

\subsection{Separation Procedure for Disjunctive Cuts}
\label{sec:chooseyhat}


We turn our attention to describing how to computationally separate our \DCs for a solution $(x^*,y^*) \in \mathcal{P}$ now.
Note that we do not necessarily need the optimal solution of the follower problem~\eqref{eq:follower} for $x=x^*$ to be able to cut off a bilevel infeasible solution $(x^*,y^*)$, as any $\hat y$ that is feasible for the follower problem  with $q(\hat y)< q(y^*)$ gives a violated \DC as described in Theorem~\ref{tm:getCut}. Thus, we implement two different strategies for separation which are described in Algorithm~\ref{alg:separation}. 

In the first one, denoted as \texttt{O}, we solve the follower problem to optimality, and use the optimal $\hat y$ in \eqref{CG-SOCP}. In the second strategy, denoted as \texttt{G}, for each feasible integer follower solution $\hat y$ with a better objective value than $q(y^*)$ obtained during solving the follower problem, we try 
to solve \eqref{CG-SOCP}. The procedure returns the first found significantly violated cut, i.e., it finds a \DC greedily. A cut $\alpha'x +\beta'y \geq \tau$ is considered to be \emph{significantly violated} by $(x^*,y^*)$ if $\tau - \alpha'x^* -\beta'y^*> \varepsilon$ for some $\varepsilon>0$. 

If $(x^*,y^*)$ is a bilevel infeasible solution satisfying integrality constraints, Algorithm~\ref{alg:separation} returns a violated cut with both strategies. 
Otherwise, i.e., if $(x^*,y^*)$ is not integer, a cut may not be obtained, because it is possible that there is no feasible $\hat y$ for the follower problem with $q(\hat y)< q(y^*)$.


\SetKwRepeat{Do}{do}{while}

\begin{algorithm}[tb]
\LinesNumbered
\SetKwInOut{Input}{Input}\SetKwInOut{Output}{Output}
\Input{A feasible \ccrHPR solution $(x^*, y^*)$, a separation \emph{strategy} \texttt{O} or \texttt{G}, a set $\mathcal{P}$} 
\Output{A significantly violated disjunctive cut or nothing} 
\BlankLine 
{
\While{the follower problem is being solved for $x=x^*$ by an enumeration based method}
{
\For{each feasible integer $\hat{y}$ with $q(\hat{y})<q(y^*)$}
{
\If{$\textit{strategy}=\texttt{G}$ or ($\textit{strategy}=\texttt{O}$ and $\hat{y}$ is optimal)}
{
solve (CG-SOCP) for $(x^*, y^*)$, $\hat{y}$ and $\mathcal{P}$\;
\If{$\tau - \alpha'x^* -\beta'y^*>\varepsilon$ }
{
\Return{$\alpha'x +\beta'y \geq \tau$}\;
}
}}}}
\caption{\texttt{separation}\label{alg:separation}}
\end{algorithm}

 \section{Solution Methods Using Disjunctive Cuts}

We now present two solution methods based on \DCs, one applicable for the general bilevel model~\eqref{bilevel}, one dedicated to a binary version of~\eqref{bilevel}.

\subsection{A Branch-and-Cut Algorithm  
\label{sec:bc}}

We propose to use the \DCs in a \BC algorithm to solve the bilevel model~\eqref{bilevel}. The \BC can be obtained by modifying any given continuous-relaxation-based \BB algorithm to solve the \HPR (assuming that there is an off-the-shelf solver for \ccrHPR that always returns an extreme optimal solution $(x^*,y^*)$ like e.g., a simplex-based \BB for a linear \ccrHPR 
\footnote{This assumption is without loss of generality, as we can outer approximate second-order conic constraints of $\mathcal{P}$ and get an extreme optimal point by a simplex method.}).

The algorithm works as follows: Use \ccrHPR as initial relaxation $\mathcal P$ at the root-node of the \BC. Whenever a solution $(x^*,y^*)$ which is integer is encountered in a \BC node, call the \DC separation. If a violated cut is found, add the cut to the set $\mathcal P$
(which also contains, e.g., variable fixing by previous branching decisions, previously added globally or locally valid \DCs, \ldots)
of the current \BC node, otherwise the solution is bilevel feasible and the incumbent can be updated. 
Note that \DCs are only locally valid except the ones from the root node, since $\mathcal{P}$ includes branching decisions.
If $\mathcal{P}$ is empty
or optimizing over $\mathcal{P}$ leads to an objective function value that is larger than the objective function value of the current incumbent, we fathom the current node.
In our implementation, we also use \DC separation for fractional $(x^*,y^*)$ as described in Section~\ref{sec:chooseyhat} for strengthening the relaxation. 

\begin{theorem}\label{th:finite}
The \BC solves the bilevel model~\eqref{bilevel} in a finite number of 
\BC-iterations under our assumptions.
\end{theorem}

\begin{proof}


First, suppose the \BC terminates, but the solution $(x^*,y^*)$ is not bilevel feasible. This is not possible, as by Theorem~\ref{tm:getCut} and the observations thereafter the \DC generation procedure finds a violated cut to cut off the integer point $(x^*,y^*)$ in this case. 

Next, suppose the \BC terminates and the solution $(x^*,y^*)$ is bilevel feasible, but not optimal. This is not possible, since by construction the \DCs never cut off any bilevel feasible solution.

Finally, suppose the \BC never terminates. This is not possible, as all variables are integer and bounded, thus there is only a finite number of nodes in the \BC tree. Moreover, this means there is also a finite number of integer points $(x^*,y^*)$, thus we solve the follower problem and $\eqref{CG-SOCP}$ a finite number of times.
The follower problem is discrete and can therefore be solved in a finite number of iterations.
\end{proof}

\subsection{An Integer Cutting Plane Algorithm}
\label{sec:intcut}

The \DCs can be directly used in a cutting plane algorithm under the following assumption. 
\begin{Assumption}\label{as:binary}
All variables in the bilevel model~\eqref{bilevel} are binary variables.
\end{Assumption}
The algorithm is detailed in Algorithm~\ref{alg:cutting}. 
It starts with the \HPR as initial relaxation of \VFR, which is solved to optimality. Then the chosen \DC separation routine (either \texttt{O} or \texttt{G}) is called to check if the obtained integer optimal solution is bilevel feasible. If not, the obtained \DC is added to the relaxation to cut off the optimal solution, and the procedure is repeated with the updated relaxation. 

Due to Assumption~\ref{as:binary} each obtained integer optimal solution is an extreme point of the convex hull of \ccrHPR, and thus due to Theorem~\ref{tm:getCut} a violated cut will be produced by the \DC separation if the solution is not bilevel feasible.
\SetKwRepeat{Do}{do}{while}
\begin{algorithm}[htb]
\LinesNumbered
\SetKwInOut{Input}{Input}\SetKwInOut{Output}{Output}
\Input{An instance of problem \eqref{bilevel} where all variables are binary} 
\Output{An optimal solution $(x^*,y^*)$} 
{
$\mathcal R \gets$ \HPR; $\mathcal P \gets$ set of feasible solutions of \ccrHPR; $\texttt{violated} \gets True$\;
\Do{\texttt{violated}}{
$\texttt{violated} \gets False$\;
solve $\mathcal R$ to optimality, let $(x^*,y^*)$ be the obtained optimal solution\;
call \texttt{separation} for $(x^*,y^*)$ and $\mathcal{P}$ with strategy \texttt{O} or \texttt{G}\; 
\If{a violated cut is found for $(x^*,y^*)$}
{
$\texttt{violated} \gets True$; add the violated cut to $\mathcal R$ and to $\mathcal P$\;
}
}
\Return{$(x^*,y^*)$}
}
\caption{\texttt{cutting plane}\label{alg:cutting}}
\end{algorithm} 
\def\usebb{} 
\let\usebb\undefined
\ifdefined\usebb
\subsection{A Branch-and-Bound Algorithm  
\label{sec:bb}}

The \BB used to solve \eqref{bilevel} follows the \BB for \MIBLOPs described in~\cite[Section 3]{fischetti2018use}. The \BB needs the auxiliary algorithm \texttt{refine} described in Algorithm~\ref{alg:refine}. For a given \HPR solution $(x^*,y^*)$, this algorithm returns the optimal solution for the restricted version of \eqref{bilevel} where $x_i=x^*_i$ for all $i \in I$. This solution is needed for the node-fathoming procedure of the \BB. 

\begin{algorithm}[h!tb]
\LinesNumbered
\SetKwInOut{Input}{Input}\SetKwInOut{Output}{Output}
\Input{An \HPR\ solution $(x^*, y^*)$} 
\Output{A bilevel-feasible solution $(\xh,\yy)$ with $\xh_i=x^*_i$ for all $i \in I$} 
\BlankLine 
{	                                                                    
	Solve the follower problem for $x=x^*$ to compute $\Phi(x^*)$\; \label{Step:Phi}
	temporarily add the following constraints to \HPR: $x_i=x^*_i$ for all $i\in I$, and $q(y) \le \Phi(x^*)$\; \label{Step:fix} 
	solve the modified \HPR\; \label{Step:solve_rHPR}
	\lIf{the modified \HPR\ is unbounded} {\Return ``The problem is unbounded''} \label{Step:unb}
	Let $(\xh,\yy)$ be the optimal solution found, and \Return $(\xh, \yy)$ \label{Step:rHPR}
}
\caption{\texttt{refine}\label{alg:refine}}
\end{algorithm} 

The \BB is based on the \ccrHPR and proceeds by using standard branching on the integer-constrained variables $x_i$, $i\in I$ and $y_j$, $j\in J$ that are fractional for the relaxation at a given node in the \BB tree. If the relaxation at a given node is infeasible, or the objective value of the relaxation is worse than the objective value of the incumbent, the node is pruned as usual. However, different to a standard \BB algorithm, once a node-solution fulfills all integrality requirements (i.e., it is a feasible solution to the \HPR), we do not directly have a new incumbent but call the procedure \texttt{BilevelNodeProcessing} instead, which is described Algorithm~\ref{alg:bb}.

\begin{algorithm}
\LinesNumbered
\SetKwInOut{Input}{Input}\SetKwInOut{Output}{Output}
\Input{An \HPR\ solution $(x^*, y^*)$} 
\Output{updated incumbent \emph{or} node is pruned \emph{or} two new children nodes} 
\BlankLine 

			compute $\Phi(x^*)$ by solving the follower problem for $x=x^*$\;\label{Step:feas1bis}
			\If {$q(y^*) \le \Phi(x^*)$ }
			{update the incumbent\;
			\Return
			} 
	
	  \eIf {all variables $x_i$, $i \in I$ are fixed by branching \label{Step:ref0} }
				{ call Algorithm~\ref{alg:refine} with $(x^*,y^*)$\; \label{Step:ref}
					possibly update the incumbent with the resulting solution $(\xh,\yy)$, if any\;
					prune the current node \label{Step:ref2}
				}
		 		{branch on any $x_i$, $i \in I$ not fixed by branching yet}

\caption{\texttt{BilevelNodeProcessing}}\label{alg:bb}
\end{algorithm}  

In this procedure, we first check if the node-solution $(x^*,y^*)$ is bilevel feasible by comparing the follower objective value obtained for $y^*$ against the optimal follower solution value for $x^*$. If the solution is bilevel feasible, we update the incumbent and prune the node. If the solution is not bilevel feasible, there are two possible cases: Either we have already fixed all integer variables $x$ by branching or not. In the latter case, we continue with branching on a non-fixed integer variable (even though the variable is already integer in the relaxation). In case we have fixed all integer variables $x$ by branching we call \texttt{refine} to find the best bilevel feasible solution for the integer-part of $x^*$. If the obtained solution is better than the incumbent, we update the incumbent. In any case, after doing this, we can prune the node.

We apply the separation routine described in Section~\ref{sec:chooseyhat} to any relaxation solution $(x^*,y^*)$ we encounter during the \BB. We note that in theory, the \DC separation for any integer $(x^*,y^*)$ within a standard \BB using the \ccrHPR would already be enough to ensure a correct algorithm for solving problem \eqref{bilevel}, but in practice, these cuts could fail to cut off a bilvel infeasible solution due to numerical issues.
\fi
\section{Computational Analysis}
\label{sec:comp}

In this section we present preliminary computational results. 

\subsection{Instances}
In our computations, we  consider the quadratic bilevel covering problem
\begin{subequations}
\label{qcbp}
\begin{align}
&\min  ~ \exSymb{c}' x + \exSymb{d}'y  \\
\st &~\exSymb{M}' x + \exSymb{N}' y  \ge \exSymb{h} \label{qcbp:pureLeaderCon} \\ 
&y  \in \arg \min \{ y' \exSymb{R} y : \exSymb{a}' x + \exSymb{b}' y  \ge \exSymb{f}, ~y \in \{0,1\}^{n_2} 
\}\\
&x \in \{0,1\}^{n_1},\label{eq:BQKP} 
\end{align}
\end{subequations}
where 
$\exSymb{c} \in \R^{           n_1}$, 
$\exSymb{d} \in \R^{           n_2}$, 
$\exSymb{M} \in \R^{m_1 \times n_1}$, 
$\exSymb{N} \in \R^{m_1 \times n_2}$, 
$\exSymb{h} \in \R^{           m_1}$, 
$\exSymb{R} =\exSymb{V}'\exSymb{V} \in \Z^{n_2 \times n_2}$, 
$\exSymb{a} \in \Z^{n_1}$, 
$\exSymb{b} \in \Z^{n_2}$, and 
$\exSymb{f} \in \Z$.
This problem can be seen as the covering-version of the quadratic bilevel knapsack problem studied by Zenarosa et al.\ in~\cite{zenarosa2021exact} (there it is studied with a quadratic non-convex leader objective function, only one leader variable and no leader constraint \eqref{qcbp:pureLeaderCon}). The linear variant of such a bilevel knapsack-problem is studied in, e.g., \cite{BrotcorneHM09,Brotcorne}. We note that~\cite{BrotcorneHM09,Brotcorne,zenarosa2021exact} use problem-specific solution approaches to solve their respective problem.
The structure of~\eqref{qcbp} allows an easy linearization of the nonlinear follower objective function using a standard McCormick-linearization to transform the problem into an \MIBLOP. Thus we can compare the performance of our algorithm against a state-of-the-art \MIBLOP-solver \MIX from Fischetti et al.~\cite{fischetti2017new} to get a first impression of whether our development of a dedicated solution approach for \IBNPs exploiting nonlinear techniques is a promising endeavour.

We generated random instances in the following way. We consider $n_1 = n_2$ for $n_1+n_2 = n\in \{20, 30, 40, 50\}$ and we study instances with no (as in~\cite{zenarosa2021exact}) and with one leader constraint \eqref{qcbp:pureLeaderCon}, so $m_1 \in \{0,1\}$. For each $n$ we create five random instances for $m_1=0$ and five random instances for $m_1=1$.
Furthermore, we chose all entries of $\exSymb{c}$, $\exSymb{d}$, $\exSymb{M}$, $\exSymb{N}$, $\exSymb{a}$, and $\exSymb{b}$ randomly from the interval $[0,99]$. The values of $\exSymb{h}$ and $\exSymb{f}$ are set to the sum of the entries of the corresponding rows in the constraint matrices divided by four.
The matrix $\exSymb{V}\in \R^{n_2 \times n_2}$ has random entries from the interval $[0,9]$.



\subsection{Computational Environment}

All experiments are executed on a single thread of an Intel Xeon E5-2670v2 machine with 2.5 GHz processor 
with a memory limit of 8 GB and a time limit of 600 seconds.
Our \BC algorithm and our cutting plane algorithm both are implemented in C++. They make use of 
IBM ILOG CPLEX 12.10 (in its default settings) as branch-and-cut framework in our \BC algorithm and as solver for $\mathcal R$ in our cutting plane algorithm.
During the \BC, CPLEX’s internal heuristics are allowed and a bilevel infeasible heuristic solution is just discarded if a violated cut cannot be obtained. For calculating the follower response $\hat y$ for a given $x^*$, we also use CPLEX. For solving \eqref{CG-SOCP}, we use MOSEK~\cite{mosek}. 

\subsection{Numerical Results}
In this preliminary computational study we use a simplified version of both the the B\&C and the cutting plane algorithm, namely we always use the initial $\mathcal{P}$, i.e., the \ccrHPR, as input for the separation of \DCs and do not update it.

While executing our B\&C algorithm, we consider four different settings for the separation of cuts. \texttt{I} and \texttt{IF} denote the settings where only integer solutions are separated and where both integer and fractional solutions are separated, respectively. For each of them, we separate the cuts using the routine \texttt{separation} with strategies \texttt{O} and \texttt{G}, which is indicated with an ``-\texttt{O}" or ``-\texttt{G}" next to the relevant setting name. The resulting four settings are \texttt{I-O}, \texttt{IF-O}, \texttt{I-G} and \texttt{IF-G}. Similarly, the cutting plane algorithm is implemented with both separation strategies, leading to the settings $\texttt{CP-O}$ and $\texttt{CP-G}$. We determine the minimum acceptable violation $\varepsilon=10^{-6}$ for our experiments.
During the integer separation of $(x^*,y^*)$, while solving the follower problem, we make use of the follower objective function value $q(y^*)$, by setting it as an upper cutoff value. This is a valid approach because a violated \DC exists only if $\Phi(x^*)<q(y^*)$. 

\begin{table}[tbp]
    \setlength{\tabcolsep}{4pt}
  \centering
  \caption{Results for the quadratic bilevel covering problem.}
    \begin{tabular}{crrrrrrrrrr}
    \hline
    \multicolumn{1}{c}{$n$} & \multicolumn{1}{c}{\texttt{Setting}} & \multicolumn{1}{c}{$t$} & \multicolumn{1}{c}{\texttt{Gap}} & \multicolumn{1}{c}{\texttt{RGap}} & \multicolumn{1}{c}{\texttt{Nodes}} & \multicolumn{1}{c}{\texttt{nICut}} & \multicolumn{1}{c}{\texttt{nFCut}} & \multicolumn{1}{c}{$t_F$} & \multicolumn{1}{c}{$t_S$} & \multicolumn{1}{c}{\texttt{nSol}} \\
    \hline
    \multirow{4}[2]{*}{20} & \texttt{I-O}     & 1.6   & 0.0   & 42.9  & 158.9 & 44.0  & 0.0   & 0.7   & 0.4   & 10 \\
          & \texttt{IF-O}     & 7.0   & 0.0   & 46.1  & 82.8  & 13.5  & 151.3 & 5.0   & 1.5   & 10 \\
          & \texttt{I-G}     & 1.1   & 0.0   & 42.6  & 192.4 & 56.7  & 0.0   & 0.2   & 0.4   & 10 \\
          & \texttt{IF-G}     & 3.3   & 0.0   & 42.1  & 102.4 & 17.3  & 183.9 & 0.7   & 2.0   & 10 \\
    \hline
    \multirow{4}[2]{*}{30} & \texttt{I-O}     & 26.1  & 0.0   & 40.4  & 2480.0 & 325.5 & 0.0   & 22.3  & 2.2   & 10 \\
          & \texttt{IF-O}     & 246.5 & 9.6   & 45.9  & 522.6 & 24.9  & 2104.1 & 216.0 & 20.3  & 8 \\
          & \texttt{I-G}     & 2.7   & 0.0   & 48.6  & 1630.6 & 226.1 & 0.0   & 0.4   & 1.7   & 10 \\
          & \texttt{IF-G}      & 55.2  & 0.0   & 39.6  & 669.6 & 29.9  & 1631.7 & 5.9   & 40.5  & 10 \\
    \hline
    \multirow{4}[2]{*}{40} & \texttt{I-O}     & 262.0 & 3.6   & 70.4  & 9209.4 & 1308.8 & 0.0   & 233.5 & 18.6  & 8 \\
          & \texttt{IF-O}     & 439.9 & 35.7  & 66.8  & 391.3 & 30.1  & 1751.9 & 390.7 & 43.8  & 4 \\
          & \texttt{I-G}     & 82.3  & 0.0   & 67.9  & 14225.5 & 1379.1 & 0.0   & 4.2   & 47.1  & 10 \\
          & \texttt{IF-G}      & 387.1 & 6.1   & 64.0  & 1039.8 & 53.4  & 3783.1 & 22.0  & 331.4 & 8 \\
    \hline
    \multirow{4}[2]{*}{50} & \texttt{I-O}     & 537.6 & 46.3  & 72.5  & 10921.1 & 1553.6 & 0.0   & 458.4 & 67.5  & 2 \\
          & \texttt{IF-O}     & 600.0 & 71.6  & 72.7  & 156.3 & 24.8  & 1272.5 & 545.2 & 51.5  & 0 \\
          & \texttt{I-G}     & 417.9 & 20.2  & 71.8  & 93621.8 & 6928.2 & 0.0   & 17.6  & 102.5 & 4 \\
          & \texttt{IF-G}      & 519.8 & 40.5  & 72.8  & 2537.6 & 56.0  & 12548.1 & 45.4  & 244.9 & 3 \\
    \hline
    \end{tabular}%
  \label{tab:aggregatedResults}%
\end{table}%

The results of the \BC algorithm are presented in Table~\ref{tab:aggregatedResults}, as averages of the problems with the same size $n$. We provide the solution time $t$ (in seconds), the optimality gap \texttt{Gap} at the end of time limit (calculated as $100 (z^*-LB)/z^*$, where $z^*$ and $LB$ are the best objective function value and the lower bound, respectively), the root gap \texttt{RGap} (calculated as $100 (z^*_R-LB_r)/z^*$, where $z^*_R$ and $LB_r$ are the best objective function value and the lower bound at the end of the root node, respectively), the number of \BC nodes \texttt{Nodes}, the numbers of integer \texttt{nICut} and fractional cuts \texttt{nFCut}, the time $t_F$ to solve the follower problems, the time $t_S$ to solve \eqref{CG-SOCP}, and the number of optimally solved instance \texttt{nSol} out of 10. \texttt{I-G} is the best performing setting in terms of solution time and final optimality gaps. Although \texttt{IF-O} and \texttt{IF-G} yield smaller trees, they are inefficient because of invoking the separation routine too often, which is computationally costly. Therefore, they are not included in further comparisons.

\begin{figure}[tbp]
\begin{center}
\includegraphics[width=0.5\textwidth]{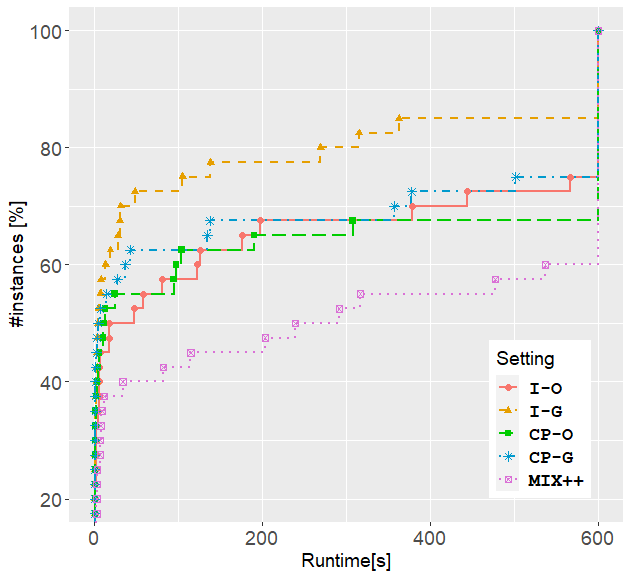}%
\includegraphics[width=0.5\textwidth]{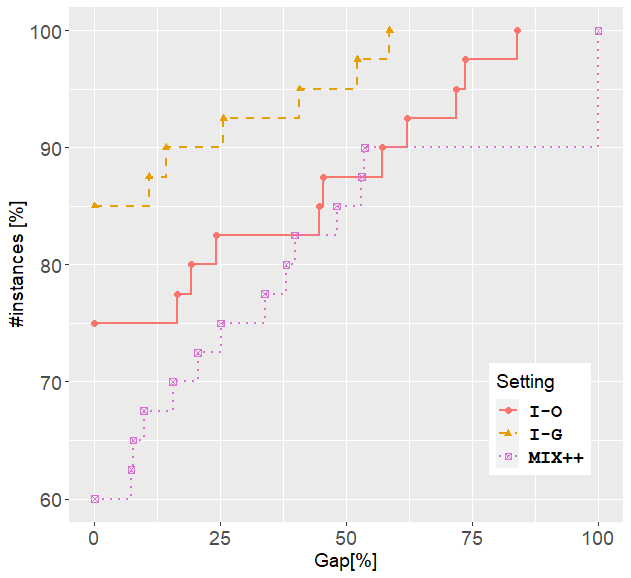}
\end{center}

\caption{Runtimes and final optimality gaps for the quadratic bilevel covering problem under our different settings and the benchmark solver \MIX.}
\label{figure:Runtime-Gap}
\end{figure}

In Figure~\ref{figure:Runtime-Gap}, we compare the B\&C results with the results obtained by the cutting plane algorithm as well as a state-of-the-art \MIBLOP solver \MIX of Fischetti et al.~\cite{fischetti2017new}, which is able to solve the linearized version of our instances. Figure~\ref{figure:Runtime-Gap} shows the cumulative distributions of the runtime and the optimality gaps at the end of the time limit. 
It can be seen that settings with \texttt{G} perform better than their \texttt{O} counterparts. While \texttt{CP-O} and \texttt{CP-G} perform close to \texttt{I-O}, they are significantly outperformed by \texttt{I-G}. The solver \MIX is also outperformed by both the cutting plane algorithm and the B\&C. 

\section{Conclusions}
In this article we showed that SOCP-based \DCs are an effective and promising methodology for solving a challenging family of discrete \BPs with a convex quadratic objective and linear constraints at the lower level.   

There are still many open questions for future research. From the computational 
perspective, dealing with multiple linking constraints at the lower level 
requires an implementation of a SOCP-based separation procedure based on 
multi-disjunctions. The implementation can also be extended to deal with 
second-order cone constraints at the upper level. Moreover, the proposed \BC 
could be enhanced by bilevel-specific preprocessing procedures, or 
bilevel-specific valid inequalities (as this has been done for \MIBLOPs in 
e.g., 
\cite{fischetti2016intersection,fischetti2017new,fischetti2018use}). 
Problem-specific strengthening 
inequalities could be used within disjunctions to obtain stronger DCs, and 
finally outer-approximation could be used as an alternative to SOCP-based 
separation. 
\bibliographystyle{splncs04}
\bibliography{disj}

\end{document}